\documentclass[12pt,a4paper,oneside]{article}
\usepackage[a4paper,left=3cm,right=3cm, top=3cm, bottom=3cm]{geometry}
\usepackage[latin1]{inputenc}
\usepackage{enumerate}
\usepackage{mathrsfs}
\usepackage[T1]{fontenc}
\usepackage {graphicx}
\usepackage{amsmath}
\usepackage{amsthm}
\usepackage{amssymb}
\usepackage{hyperref}
\theoremstyle{plain}
\newtheorem{thm}{Theorem}[section]  

\newtheorem{prop}[thm]{Proposition}
\theoremstyle{definition}
\newtheorem{defn}[thm]{Definition}
\newtheorem{example}[thm]{Example}
\theoremstyle{remark}
\newtheorem{remark}[thm]{Remark}

\newcommand{\vertiii}[1]{{\left\vert\kern-0.25ex\left\vert\kern-0.25ex\left\vert #1 
    \right\vert\kern-0.25ex\right\vert\kern-0.25ex\right\vert}}
\title{New examples of non-commutative Valdivia compact spaces}
\author{Jacopo Somaglia\footnote{Research was supported in part by the 
Universit\`a degli Studi of Milano (Italy), in part by the Gruppo Nazionale per l'Analisi Matematica, la Probabilit\`a e le loro Applicazioni (GNAMPA) of the Istituto Nazionale di Alta Matematica (INdAM) of Italy and in part by the research grant GA\v{C}R 17-00941S.}}
\date{}
\begin{document}

\maketitle

\begin{abstract}
\noindent
The aim of this note is to characterize trees, endowed with coarse wedge topology, that have a retractional skeleton. We use this characterization to provide new examples of non-commutative Valdivia compact spaces that are not Valdivia.\\

\noindent
\textit{MSC:} 54D30, 54C15, 54G20, 54F05\\

\noindent
\textit{Keywords:} retractional skeleton, Valdivia compact spaces, tree, coarse wedge topology, bounded topology
 
\end{abstract}

\section{Introduction}

The class of Valdivia compact spaces plays an important role in the study of nonseparable Banach spaces. It has been investigated for example in \cite{DeviGode} or more recently in \cite{Kalenda3}. For a detailed survey of this subject we refer to \cite{Kalenda2}. A related class of compact spaces is the class of the non-commutative Valdivia compacta. As in \cite{Cuth1}, for non-commutative Valdivia compacta we mean the class of compact spaces with retractional skeleton. These two classes share several topological properties (see for example \cite{Cuth1} and \cite{Soma}). The Banach spaces associated to Valdivia compacta $ \langle$respectively non-commutative Valdivia compacta$\rangle$ are called Plichko spaces $\langle$respectively Banach spaces with projectional skeletons$\rangle$. These two classes share several structural and geometrical properties (see \cite{Kubis1} and \cite{Cuth2}). The definition of retractional skeleton was introduced by Kubi\'{s} and Michalewski in \cite{KubisMicha}. In the same paper the authors proved that a compact space is Valdivia if and only if it has a commutative retractional skeleton. The two classes do not agree: the ordinal space $[0,\omega_2]$, endowed with the interval topology, is an easy well-known example of a non-commutative Valdivia compact space that is not Valdivia (see \cite[Example 6.4]{KubisMicha}). Nevertheless, since $[0,\omega_2]$ is involved in any such example, the following question seems to be of interest:
\begin{center}
\textit{Let $X$ be a non-commutative Valdivia compact space that does not contain any copy of the ordinal space $[0,\omega_2]$. Is $X$ necessarily Valdivia?}
\end{center}
The present paper answers in the negative the question above. That can be done via a characterization of trees, endowed with coarse wedge topology, with retractional skeleton. That characterization can have an independent interest.\\
The coarse wedge topology (also known as path topology) has been studied in detail in \cite{Todo3} for providing a minimal size counterexample to preservation of Lindel\"{o}f property in products. In \cite{Todo1} is contained an example of tree endowed with the same topology which is a first countable Corson compact space with no dense metrizable subspace.\\
The paper is organized as follows. In the remaining part of the introductory section, notations and basic notions addressed in this paper are given. In section 2 we provide basic definitions on trees and we recall the basic results on the coarse wedge topology, finally we recall some useful results about Valdivia and non-commutative Valdivia compacta classes. Section 3 is devoted to proving the characterization of trees with retractional skeleton. In section 4 we investigate the relations between trees and Valdivia compact spaces, in particular we present an example that gives a negative answer to the question above.\\
We denote with $\omega$ the set of natural numbers (including 0) with the usual order. Given a set $X$ we denote by $|X|$ the cardinality of the set $X$, by $[X]^{\leq \omega}$ ($[X]^{<\omega}$) the family of all countable (finite) subsets of $X$.\\
All the topological spaces are assumed to be Hausdorff and completely regular. Given a topological space $X$ we denote by $\overline{A}$ the closure of $A\subset X$. We say that $A\subset X$ is countably closed if $\overline{C}\subset A$ for every $C\in [A]^{\leq \omega}$. A topological space $X$ is a Fr\'{e}chet-Urysohn space if for every $A\subset X$ and $x\in \overline{A}$ there is a sequence $\{x_{n}\}_{n\in\omega}\subset A$ such that $x_{n}\to x$. $\beta X$ denotes the \v{C}ech-Stone compactification of $X$.\\ 

\section{Basic notions}

In the first part of this section we recall some basic definitions about trees and the topology on them. The second part is devoted to giving the basic definitions and properties of classes of compact spaces which we are interested in.

\begin{defn}
A \textit{tree} is a partially ordered set $(T,\leq)$, such that for each $x\in T$ the set $\{y\in T:y<x\}$ of all predecessors of $x$ is well-ordered by $<$.
\end{defn}

\noindent
If a tree has only one minimal member, it is said to be \textit{rooted} and the minimal member is called the \textit{root} of the tree, in this case we will denote the root by $0$. Maximal chains are called \textit{branches}. Given a tree $T$ we will use the following notations:
\begin{itemize}
\item the height of $x$ in $T$, denoted by ht$(x,T)$, is the order type of $\{y\in T: y<x \}$;
\item for each ordinal $\alpha$, the $\alpha$-th level of $T$, is the set defined by $\operatorname{Lev}_{\alpha}(T)=\{x\in T:\mbox{ht}(x,T)=\alpha\}$;
\item by cf$(x)$ we indicate the cofinality of ht$(x,T)$;
\item the height of $T$, denoted by ht$(T)$, is the least $\alpha$ such that $\operatorname{Lev}_{\alpha}(T)=\emptyset$;
\item we denote by ims$(x)=\{y\in T:x<y \,\& \,\mbox{ht}(y,T)=\mbox{ht}(x,T)+1\}$ the set of the immediate successors of $x$.
\end{itemize}

For $t\in T$ we put $V_t=\{s\in T:s\geq t\}$ and $\hat{t}=\{s\in T:s\leq t\}$. In this work we consider $T$ endowed with the \textit{coarse wedge topology}. Now we recall the definition of such a topology and its properties, we refer to \cite{Todo3} and \cite{Nyikos2} for the details. The coarse wedge topology on a tree $T$ is the one whose subbase is the set of all $V_t$ and their complements, where $t$ is  either minimal or on a successor level. Let us describe a local base at each point $t\in T$. If ht$(t,T)$ is a successor or $t$ is the minimal element, a local base at $t$ is formed by sets of the form

\begin{equation*}
W_{t}^{F}=V_t\setminus \bigcup\{V_s:s\in F\},
\end{equation*}

\noindent
where $F$ is a finite set of immediate successors of $t$. In case ht$(t,T)$ is limit, a local base at $t$ is formed by sets of the form

\begin{equation*}
W_{s}^{F}=V_s\setminus \bigcup\{V_r:r\in F\},
\end{equation*}

\noindent
where $s<t$, ht$(s,T)$ is a successor and $F$ is a finite set of immediate successor of $t$. From now on every tree is considered endowed with the coarse wedge topology. We are interested in examples of compact spaces, therefore we recall the characterization of the trees which are compact in the coarse wedge topology.

\begin{defn}
A tree is called \textit{chain complete} if every chain has a supremum.
\end{defn}

\begin{thm}\cite[Corollary 3.5]{Nyikos2},\label{CoarseCompact}
A tree  $T $ is compact Hausdorff in the coarse wedge topology if and only if $T$ is chain complete and has finitely many minimal elements.
\end{thm}

From now on we will consider only chain complete trees with a unique minimal element. We observe that if $T$ is as in the previous theorem, it is a zero dimensional compact space. Let us now give the definition and a characterization of Valdivia compact spaces.

\begin{defn}
For any set $\Gamma$ we put $\Sigma(\Gamma)=\{x\in[0,1]^{\Gamma}:|\{\gamma\in\Gamma:\, x(\gamma)\neq 0\}|\leq\omega_0\}$. Let $K$ be a compact space, we say that $A\subset K$ is a $\Sigma$-subset of $K$ if there is a homeomorphic injection $h$ of $K$ into some $[0,1]^\Gamma$ such that $h(A)=h(K)\cap \Sigma(\Gamma)$. $K$ is called \textit{Valdivia compact space} if it has a dense $\Sigma$-subset.
\end{defn}

\noindent
A family of sets $\mathcal{U}$ is $T_0$-separating in $X$ if for every two distinct elements $x,y\in X$ there is $U\in\mathcal{U}$ satisfying $|\{x,y\}\cap U|=1$. A family $\mathcal{U}$ is point countable on $D\subset X$ if

\begin{equation*}
|\{U\in\mathcal{U}:\,x\in U\}|\leq \omega
\end{equation*}

\noindent
for every $x\in D$. Given $x\in T$ we indicate by $\mathcal{U}(x)=\{U\in\mathcal{U}:x\in U\}$. We recall a useful characterization of Valdivia compact spaces.

\begin{thm}\cite[Proposition 1.9]{Kalenda2},\label{T0SeparatingFamily}
Let $K$ be a compact space and $D$ be a dense subset of $K$. Then $D$ is a $\Sigma$-subset of $K$ if and only if there is a family $\mathcal{U}$ of open $F_{\sigma}$ subsets of $K$ which is $T_0$-separating and $D$ is the set of point countability, i.e., $D=\{x\in K: \,\mathcal{U}(x)\mbox{ is countable}\}$.
\end{thm}

It is easy to see that the family $\mathcal{U}$ in the previous theorem may consist of sets from a given base. Therefore, if $K$ is zero-dimensional, $\mathcal{U}$ may consists of clopen sets, see \cite[Theorem 19.11]{KKLP} for the details. Now let us recall the notion of retractional skeleton which is used to define a non-commutative counterpart of Valdivia compact spaces.

\begin{defn}\label{defnRetr}
A \textit{retractional skeleton} in a compact space $K$ is a family of continuous retractions $\{r_{s}\}_{s\in\Gamma}$, indexed by an up-directed partially ordered set $\Gamma$, such that:
\begin{enumerate}[$(i)$]
\item $r_{s}[K]$ is a metrizable compact space for each $s\in\Gamma$,
\item $s,t\in\Gamma$, $s\leq t$ then $r_{s}=r_t\circ r_s=r_s\circ r_t$,
\item given $s_{0}\leq s_{1}\leq...$ in $\Gamma$, $t=\sup_{n\in\omega}s_{n}$ exists and $r_{t}(x)=\lim_{n\to \infty}r_{s_{n}}(x)$ for every $x\in K$,
\item for every $x\in K$, $x=\lim_{s\in\Gamma}r_{s}(x)$.
\end{enumerate}
We say  that $D=\bigcup_{s\in\Gamma}r_{s}[K]$ is the set induced by the retractional skeleton $\{r_{s}\}_{s\in\Gamma}$ in $K$.
\end{defn}

By \cite{KubisMicha} $K$ is Valdivia if and only if it has a commutative retractional skeleton. More precisely, a dense set $D\subset K$ is a $\Sigma$-subset if and only if it is induced by a commutative retractional skeleton. Compact spaces having a retractional skeleton (not necessarily commutative) are called, following \cite{Cuth1}, non-commutative Valdivia compact spaces. We recall some useful and well-known results about retractional skeletons.

\begin{thm}\cite[Theorem 32]{Kubis1}\label{PropRetr}
Assume $D$ is induced by a retractional skeleton in a compact space $K$. Then:
\begin{enumerate}[$(i)$]
\item $D$ is dense in $K$ and for every countable set $A\subset D$, $\overline{A}$ is metrizable and contained in $D$.
\item $D$ is a Fr\'{e}chet-Urysohn space.
\item $D$ is a normal space and $K=\beta D$.
\end{enumerate}
\end{thm}

In particular we observe that, given a retractional skeleton in a compact space $X$, its induced subset $D$ is countably compact.

\section{Trees and Retractional Skeletons}

Since we are interested in compact Hausdorff spaces, by Theorem \ref{CoarseCompact} we assume that every tree is rooted, chain complete and endowed with the coarse wedge topology.\\
We start by stating the main result of this section, which provides a characterization of trees with retractional skeleton.

\begin{thm}\label{CarattTree}
Let $T$ be a tree. $T$ has a retractional skeleton if and only if it satisfies the following condition:
\begin{equation}\tag{$*$}
\mbox{$t$ has at most finitely many immediate successors whenever cf$(t)\geq\omega_1$.}
\end{equation}
\end{thm}

We observe that the "only if" part follows from the following proposition:

\begin{prop}\label{Proponlyif}
Let $T$ be a tree. If $T$ has a retractional skeleton then it satisfies $(*)$ and, moreover, the induced subset is $D=\{t\in T:\; \mbox{cf}(t)\leq\omega\}.$ 
\end{prop}

\begin{proof}
Let $D$ be the induced subset of a retractional skeleton on $T$. Let $t\in T$ in a successor level, we want to prove that $t\in D$. Two cases are possible:
\begin{itemize}
\item suppose that ims$(t)$ is finite, then $t$ is isolated and therefore, since $D$ is dense, we have $t\in D$;
\item suppose there exists an infinite subset $\{t_n\}_{n\in\omega}\subset $ ims$(t)$. For every $n\in\omega$, by the density of $D$, there exists $s_n\in V_{t_n}\cap D$. We observe that the sequence $\{s_n\}_{n\in\omega}$ converges to $t$, hence, since $D$ is countably closed, we have $t\in D$. 
\end{itemize}
Since $D$ is countably closed $t\in D$ whenever cf$(t)=\omega$. Since $D$ is a Fr\'{e}chet-Urysohn space, if cf$(t)\geq\omega_1$ we have that $t\notin D$. This implies that ims$(t)$ is finite whenever cf$(t)\geq \omega_1$.
\end{proof}

The remaining part of this section is devoted to prove the "if part" of Theorem \ref{CarattTree}. The strategy is to find a suitable class of subsets of the tree that play the role of the set of indices of the retractional skeleton.

\begin{defn}
Let $T$ be a tree that satisfies $(*)$. Define the following mappings:
\begin{itemize}
\item $\wedge: T\times T\to T$ such that $t\wedge s=\max (\hat{s}\cap \hat{t})$, for every $s,t\in T$.
\item $\varphi: T\to [T]^{\leq \omega}$ such that $\varphi(t)=\{0\}$ if cf$(t)\neq \omega$ and $\varphi(t)=\{t_n\}_{n\in\omega}$ if cf$(t)=\omega$, where $\{t_n\}_{n\in\omega}\subset T$ is a sequence that converges to $t$, $t_n< t$ for every $n\in\omega$ and $t_n$ belongs to a successor level for every $n\in\omega$.
\item $\psi: T\to [T]^{<\omega}$ such that $\psi(t)=\{0\}$ if cf$(t)<\omega_1$, $\psi(t)=\mbox{ims}(t)$ if cf$(t)\geq \omega_1$.
\end{itemize}
\end{defn}

We denote by $\mathcal{A}(T)$ the class of subsets $A$ of a tree $T$ that satisfy the following properties:
\begin{enumerate}[$(a)$]
\item $A$ is countable,
\item $0\in A$,
\item if $s,t\in A$, then $s\wedge t\in A$,
\item if $t\in A$, then $\varphi(t)\subset A$, 
\item if $t\in A$, then $\psi(t)\subset A$.
\end{enumerate}

\begin{prop}\label{PropeExtension}
Let $T$ be a tree that satisfies $(*)$ and $S$ be a countable subset of $T$, then there exists an $A\in \mathcal{A}(T)$ such that $S\subset A$.
\end{prop}

\begin{proof}
We will use an induction argument, consider $S_0=S\cup \{0\}$ and 
\begin{equation*}
S_{n+1}=S_n\cup \wedge(S_n\times S_n)\cup\psi(S_n)\cup \varphi(S_n).
\end{equation*} 
Define $A=\bigcup_{n\in\omega}S_n$, it remains to prove that $A\in\mathcal{A}(T)$:
\begin{enumerate}[$(a)$]
\item given a countable set $N$, we have that $\wedge(N\times N), \psi(N), \varphi(N)$ are countable. This implies that $A$ is countable.
\item $0\in S_0\subset A$;
\item let $s,t\in A$, then there exists $n\in \omega$ such that $s,t\in S_n$, therefore $s\wedge t\in S_{n+1}\subset A$;
\item let $t\in A$, then there exists $n\in\omega$ such that $t\in S_n$, therefore $\varphi(t)\subset S_{n+1}\subset A $;
\item if $t\in A$, then there exists $n\in\omega$ such that $t\in S_n$, therefore $\psi(t)\subset S_{n+1}\subset A $.
\end{enumerate}
This completes the proof.
\end{proof}

It follows from Proposition \ref{PropeExtension} that $\mathcal{A}(T)$, ordered by inclusion, is an up-directed partially ordered set. Moreover, it is obvious that  $\mathcal{A}(T)$ is $\sigma$-complete.
\\
Now we are going to list some properties of $\overline{A}$, with $A\in \mathcal{A}(T)$ and the closure is taken in $T$ with respect to the coarse wedge topology.

\begin{prop}\label{propclosure}
Let $T$ be a tree that satisfies $(*)$. Let $A\in \mathcal{A}(T)$ and $D=\{t\in T:\,\mbox{cf}(t)\leq\omega\}$, then the following properties hold:
\begin{enumerate}
\item $\overline{A}$ is separable,
\item $0 \in\overline{A}$,
\item if $t\in\overline{A}$ and cf$(t)=\omega$, then there exists a sequence $\{t_n\}_{n\in\omega}\subset A$ that converges to $t$, $t_n<t$ for every $n\in\omega$ and each element of $\{t_n\}_{n\in\omega}$ belongs to a successor level, 
\item if $t\in \overline{A}\setminus A$, then cf$(t)=\omega$,
\item if $t,s\in\overline{A}$, then $s\wedge t\in \overline{A}$,
\item if $t\in\overline{A}$, then $\psi(t)\subset \overline{A}$,
\item $\overline{A}\cap D=\overline{A\cap D}$.
\end{enumerate}
\end{prop}

\begin{proof}
$(1)$ and $(2)$ are clear.\\
\\
$(3)$ If cf$(t)=\omega$ and $t\in A$, the assertion follows by taking $\varphi(t)$. Suppose that $t\in \overline{A}\setminus A$ and cf$(t)=\omega$.\\
The first step is to prove that for every $s<t$, with $s$ on a successor level, there exists $w\in A$ such that $s\leq w<t$. In order to do this, we take an open neighborhood $W_{s}^{F}$ such that if $x\in W_{s}^{F}\cap A$, then either $x< t$ or $x$ and $t$ are incomparable. Since $t\in\overline{A}\setminus A$, such a neighborhood exists. Let $x\in W_{s}^{F}\cap A$, if $x<t$, then we take $w=x$. Suppose now that $x\in W^{F}_{s}\cap A$ and $x$ and $t$ are incomparable. Let $w=x\wedge t$ and consider $W_{r}^{F}$ where $w\leq r < t$ and $r$ is on a successor level, hence there exists $w_1\in W_{r}^{F}\cap A$ such that $w=x\wedge w_1$. Therefore $w\in A$ and $s\leq w< t$.\\
So, we can find a sequence $\{s_n\}_{n\in\omega}\subset T$ that converges to $t$, each $s_n$ belongs to a successor level and $s_n<t$ for every $n\in\omega$. By the previous consideration there exists $w_n\in A$ such that $s_n\leq w_n<t$ for every $n\in\omega$. The assertion follows by considering three different cases:
\begin{itemize}
\item if $w_n$ is on a successor level, then $t_n=w_n$;
\item if cf$(w_n)=\omega$, then there exists $t_n\in\varphi(w_n)$ such that $s_n<t_n<t $;
\item if cf$(w_n)\geq\omega_1$, then there exists $t_n\in\psi(w_n)$ such that $s_n<t_n<t $.
\end{itemize}

$(4)$ Suppose that $t\in\overline{A}$ and $t$ on a successor level. If $t$ is isolated, then obviously $t\in A$. So, suppose that $t$ is an accumulation point. Let $x\in V_t\cap A$ and $y\in W_{t}^{F}\cap A$, where $F\subset \operatorname{ims}(t)$ and $x\notin W_{t}^{F}$. Since $t=x\wedge y$ and $A\in \mathcal{A}(T)$, we have $t\in A$. On the other hand, let $t\in \overline{A}$ and cf$(t)\geq \omega_1$. Fix some $s<t$ on a successor level and set $F=$ims$(t)$. Since $A$ is countable, we have that $A\cap W_{s}^{F}\setminus\{t\}=\{s_i\}_{i\in\omega}$. Let $\overline{s}_i=t\wedge s_i$, and let $\overline{s}=\sup\{\overline{s}_i\}$. Since cf$(t)\geq \omega_1$, we have $\overline{s}<t$. Hence $W_{\overline{s}+1}^{F}\cap A=\{t\}$, where $\overline{s}+1\in \mbox{ims}(\overline{s})$ and $\overline{s}+1< t$. Therefore $t\in A$. We conclude that if $t\in \overline{A}\setminus A$, then cf$(t)=\omega.$\\
\\
$(5)$ Let $s,t\in \overline{A}$, we may assume without loss of generality that $s$ and $t$ are incomparable. Let $s_0,t_0\in$ ims$(s\wedge t)$ such that $s_0\leq s$ and $t_0\leq t$. Thus we have that $s\in V_{s_0}$, $t\in V_{t_0}$ and $V_{s_0}\cap V_{t_0}=\emptyset$. Therefore there are $s_1,t_1\in A$ such that $s_1\in V_{s_0}$ and $t_1 \in V_{t_0}$, hence $t\wedge s=t_1\wedge s_1\in A$.\\
\\
$(6)$ It follows from point $(4)$.\\
\\
$(7)$ Since $D$ is countably closed and $A\cap D$ is countable, we have $\overline{A\cap D}\subset D$, moreover, since $A\cap D\subset A$, we have $\overline{A\cap D}\subset \overline{A}$; therefore $\overline{A\cap D}\subset \overline{A}\cap D$. On the other hand, if $t\in \overline{A}\setminus A$, by $(3)$ and $(4)$, we have cf$(t)=\omega$, hence, by definition of $D$, $t\in D$ which implies $\overline{A}\cap D\subset \overline{A\cap D}$.
\end{proof}

\begin{prop}\label{PropContinuity}
Let $T$ be a tree that satisfies $(*)$, $A\in\mathcal{A}(T)$ and $r_A:T\to T$ defined by $r_A(t)=\max\{\hat{t}\cap \overline{A}\cap D\}$. Then $r_A$ is a continuous retraction with range $\overline{A}\cap D$.
\end{prop}

\begin{proof}
We observe that by point $(7)$ of Proposition \ref{propclosure}, the mapping $r_A$ is well-defined for every $A\in\mathcal{A}(T)$. Now we are going to prove the continuity of each $r_A$.\\
Let $t\in T$:
\begin{itemize}
\item if $t\notin \overline{A\cap D}$, there exists an open neighborhood $W_{s}^{F}$ of $t$ that does not intersect $\overline{A\cap D}$. Since $r_A$ is constant on $W_{s}^{F}$, $r_A$ is continuous at $t$;
\item if $t\in \overline{A\cap D}$, then we consider three different cases:
\begin{enumerate}
\item suppose that $t$ is on a successor level and ims$(t)$ is finite, then $t$ is an isolated point, therefore $r_A$ is continuous at $t$;
\item suppose that $t$ is on a successor level and ims$(t)$ is infinite. Let $W_{t}^{F}$ be an open neighborhood of $r_A(t)=t$. Observing that $r_A(s)\geq r_A(t)$ for every $s\in W_{t}^{F}$ and if, $y\in F$, then $z\notin W_{t}^{F}$ for every $z\in V_y$, we obtain $r_A(W_{t}^{F})\subset W_{t}^{F}$. Therefore $r_A$ is continuous at $t$;
\item suppose that cf$(t)=\omega$. Let $W_{s}^{F}$ be an open neighborhood of $r_A(t)=t$, hence, by the point $(3)$ of Proposition \ref{propclosure}, there exists a successor point $w\in A$ such that $s\leq w<t$. Hence $r_A(W_{w}^{F})\subset W_{s}^{F}$. Therefore $r_A$ is continuous at $t$.
\end{enumerate}
\end{itemize}
This proves the continuity.
\end{proof}

Now we can provide the "if part" of Theorem \ref{CarattTree}:

\begin{proof}[Proof of Theorem \ref{CarattTree}] Suppose that $T$ satisfies condition $(*)$, consider the family of retractions $\{r_A\}_{A\in\mathcal{A}(T)}$ as in Proposition \ref{PropContinuity} and $D=\{t\in T:\,\mbox{cf}(t)\leq\omega\}$. We have to prove $(i)-(iv)$ of Definition \ref{defnRetr}.\\
\\
$(i)$ Let $A\in\mathcal{A}(T)$, we have $r_A[T]=\overline{A}\cap D$. Since $\overline{A}\cap D$ is closed and it has at most countable chains, by  \cite[Theorem 2.8]{Nyikos1}, we have that $r_A[T]$ is a Corson compact space. Moreover, since $\overline{A}$ is separable, we have that $r_A[T]$ is metrizable.\\
\\
$(ii)$ Let $A,B\in\mathcal{A}(T)$ such that $A\leq B$, then for every $t\in T$ we have $r_A(t)\leq r_{B}(t)\leq t$. We observe that $r_A(t)\in \overline{A}\cap D\subset\overline{B}\cap D$, therefore $r_A(t)=r_B(r_A(t))$. If $r_B(t)\in \overline{A}$, then $r_A(t)=r_A(r_B(t))$. On the other hand if $r_B(t)>r_A(t)$, then there are no points $s\in\overline{A}\cap D$ such that $r_A(t)<s< r_{B}(t)$. Hence $r_A(t)=r_A(r_B(t))$.\\
\\
$(iii)$ Consider $A_1\leq A_2\leq...\leq A_n\leq...$ and $A=\sup_{n\in\omega}A_n$. Since $\overline{A_n}\cap D\subset\overline{A_m}\cap D\subset\overline{A}\cap D$ if $n\leq m$, we have, for every $t\in T$, $r_{A_n}(t)\leq r_{A_m}(t)\leq r_{A}(t)$. Hence $\lim_{n\in\omega}r_{A_n}(t)=\sup_{n\in\omega}r_{A_n}(t)$. Suppose, by contradiction, $\sup_{n\in\omega}r_{A_n}(t)<s\leq r_A(t)$ for some $s$ in a successor level. Since $V_s$ is an open neighborhood of $r_A(t)$ and $A=\bigcup_{n\in\omega}A_n$ is dense in $\overline{A}$, there exists $n\in\omega$ and $s_1\in A_n\cap V_s$. Consider $W_{s}^{F}\subset V_s$ such that $s_1\notin W_{s}^{F}$. By the same reasons of above there exists $m\geq n$ and $s_2\in A_m$ such that $s_2\in W_{s}^{F}$. Thus $s_0=s_1\wedge s_2$ belongs to $A_m$ and $s\leq s_0\leq t$. This implies $r_{A_m}(t)\geq s>\sup_{n\in\omega}r_{A_n}(t)$. A contradiction. \\
\\
$(iv)$ Let $t\in T$, we split in two cases:
\begin{itemize}
\item suppose cf$(t)\geq\omega_1$, then for every open set $t\in W_{s}^{F}$ there exists a successor point $w<t$ and $w\in W_{s}^{F}$. Hence by Proposition \ref{PropeExtension} there exists $A\in\mathcal{A}(T) $ such that $w\in A$. Hence $r_B(t)\in  W_{s}^{F}$ for every $B\supseteq A$;
\item suppose cf$(t)\leq\omega$, then by Proposition \ref{PropeExtension} there exists $A\in\mathcal{A}(T) $ such that $t\in A$. Hence, $r_B(t)=t$ for every $B\supseteq A$.
\end{itemize}
Therefore $\{r_A\}_{A\in\mathcal{A}(T)}$ is a retractional skeleton on $T$.
\end{proof}

\section{Valdivia compact trees}

In this section we investigate the relations between Valdivia compact spaces and trees. We start by proving that every tree of height at most $\omega_1 +1$ is Valdivia. On the other hand if a tree is Valdivia then its height must be less than $\omega_2 +1$.

\begin{thm}
Let $T$ be a tree, then the following assertions hold:
\begin{enumerate}[$(1)$]
\item if ht$(T)\leq \omega_1 +1$, then $T$ is a Valdivia compact space;
\item if $T$ is Valdivia, then ht$(T)<\omega_2 +1$.
\end{enumerate}
\end{thm}

\begin{proof}
$(1)$ We will use Theorem \ref{T0SeparatingFamily} to prove that $T$ is Valdivia. Let $D=\{t\in T:\mbox{cf}(t)\leq\omega\}$, $D$ is dense and since ht$(T)\leq \omega_1 +1$, if $t\in D$, then ht$(t,T)<\omega_1$. Let $\mathcal{U}=\{V_t:t\in D\;\&\; \mbox{ht}(t,T)\mbox{ is a successor ordinal}\}$. We want to prove that the family $\mathcal{U}$ is point countable on $D$. Let $t\in D$ and $s\in T$ such that $V_s\in\mathcal{U}$. We consider three cases:
\begin{itemize}
\item $s,t$ are incomparable, then $V_s\notin \mathcal{U}(t)$;
\item $s<t$, then $V_s\in \mathcal{U}(t)$;
\item $t<s$, then $V_s\notin \mathcal{U}(t)$.
\end{itemize}
Hence, since if $t\in D$ then ht$(t,T)<\omega_1$, we have $|\mathcal{U}(t)|\leq \omega$ if $t\in D$. This proves that $\mathcal{U}$ is point countable on $D$. It remains to prove that $\mathcal{U}$ is $T_0$-separating. Let $s,t\in T$ be distinct elements, we consider two cases:
\begin{itemize}
\item ht$(s,T)$ is a successor ordinal. In the case that $s$ and $t$ are incomparable or $t<s$, since $V_s\in \mathcal{U}$ and $t\notin V_s$, we have $|V_s\cap \{s,t\}|=1$. Suppose that $s<t$, then there exists an element $r$ such that ht$(r,T)$ is a successor ordinal and $s<r\leq t$. Hence $V_r \in \mathcal{U}$, $t\in V_r$ and $s\notin V_r$;
\item ht$(s,T)$ and ht$(t,T)$ are limit ordinals. If $t<s$ there exists $r\in T$ such that ht$(r,T)$ is a successor ordinal and $t<r<s$. Therefore $V_r\in \mathcal{U}$ and $|V_r\cap \{s,t\}|=1$. Otherwise, if $s$ and $t$ are incomparable there exists $r\in T$ such that ht$(r,T)$ is a successor ordinal, $r<t$ and $r$ and $s$ are incomparable. Hence $|V_r\cap \{s,t\}|=1$.
\end{itemize}
$(2)$ Since $T$ is Valdivia, by Proposition \ref{Proponlyif} we have that $D=\{t\in T: \mbox{cf}(t)\leq\omega\}$ is the $\Sigma$-subset of $T$. Suppose, by contradiction, that ht$(T)\geq \omega_2 +1$. Then there exists $t\in T$ such that ht$(t,T)=\omega_2$. This means that $\hat{t}$, with the subspace topology, is homeomorphic to $[0,\omega_2]$, with interval topology. Since $\hat{t}\cap D$ is dense in $\hat{t}$, we have that $\hat{t}$ would be a Valdivia compact space with the subspace topology. This is a contradiction because $[0,\omega_2]$ is not Valdivia.\\
\\
This concludes the proof.
\end{proof}

In the rest of this section we provide an example of tree with retractional skeleton that gives a negative answer to the question posed in the introduction.\\
Let $X=\prod_{\alpha<\omega_1}X_{\alpha}$, with $|X_\alpha|\geq 2$ for every $\alpha<\omega_1$, endowed with the topology whose basis is the collection of all sets $V(g,\alpha)=\{f\in X:g\upharpoonright \alpha=f\upharpoonright \alpha\}$, where $g\in X$ and $\alpha<\omega_1$. This is a slight generalization of the generalized Baire space endowed with the \textit{bounded topology}. We refer to \cite{AndrettaMottoRos} for a detailed reference in this field.\\
We observe that the sets $V(g,\alpha)$ are clopen and that already the sets $V(g,\alpha)$, $\alpha$ successor ordinal, form a basis of the topology.

\begin{prop}\label{BaireCategory}
Let $\{A_\xi\}_{\xi<\omega_1}$ be a family of open dense subsets of $X$, then $\bigcap_{\xi<\omega_1}A_{\xi}$ is dense in $X$. 
\end{prop}

\begin{proof}
Let $U$ be an open subset of $X$. By transfinite induction on $\xi<\omega_1$ we define a transfinite sequence of elements $\{f_{\xi}\}_{\xi<\omega_1}\subset X$ and a transfinite increasing sequence of ordinals $\{\alpha_{\xi}\}_{\xi<\omega_1}$ such that $f_{\xi_1}\upharpoonright \alpha_{\xi_1}=f_{\xi_2}\upharpoonright \alpha_{\xi_1}$ whenever $\xi_1<\xi_2$. We define those sequences in the following way: 
\begin{itemize}
\item since $U\cap A_0$ is open, there exist $f_0\in X$ and $\alpha_0<\omega_1$ such that $V(f_0,\alpha_0)\subset U\cap A_{0}$;
\item $\xi=\gamma+1$. Since $A_{\xi}\cap V(f_{\gamma},\alpha_{\gamma})$ is an open non-empty subspace of $X$, there exist $f_{\gamma+1}$ and $\alpha_{\gamma +1}$ such that $V(f_{\gamma+1},\alpha_{\gamma+1})\subset A_{\xi}\cap V(f_{\gamma},\alpha_{\gamma})$ and $f_{\gamma+1}\upharpoonright \alpha_{\gamma}=f_{\gamma}\upharpoonright \alpha_{\gamma}$;
\item $\xi$ is limit. Let $\alpha_{\xi}=\sup_{\gamma<\xi}\alpha_{\gamma}$ and  $f_{\xi}\in X$ such that $f_{\xi}\upharpoonright \alpha_{\gamma}=f_{\gamma}\upharpoonright \alpha_{\gamma}$ for every $\gamma<\xi$. Therefore we have that $V(f_{\xi},\alpha_{\xi})\subset \bigcap_{\gamma<\xi}A_{\gamma}$.
\end{itemize}
Finally, let $f_{\omega_1}\in X$ defined by $f_{\omega_1}\upharpoonright \alpha_{\xi}=f_{\xi}\upharpoonright \alpha_{\xi}$, for every $\xi<\omega_1$. $f_{\omega_1}$ belongs to $U$ and to $A_{\xi}$ for every $\xi<\omega_1$. This gives us the assertion.
\end{proof}

We recall that, given a cardinal $\kappa$ and an ordinal $\alpha$, the full $\kappa$-ary tree of height $\alpha$ is the tree of all transfinite sequences $f:\beta\to \kappa$, for some ordinal $\beta<\alpha$, with the following order: $f\leq g$ if and only if $\operatorname{dom} (f)\subset \operatorname{dom} (g)$ and $g\upharpoonright \operatorname{dom} (f)=f$.

\begin{example}
Let $T$ be a subtree of the binary tree of height $\omega_1 +2$ such that for every $t\in \operatorname{Lev}_{\omega_1}(T)$ there exists a unique immediate successor, we denote the immediate successor of $t\in  \operatorname{Lev}_{\omega_1}(T)$ by $t+1$. Consider $T$ endowed with the coarse wedge topology. We have the following:
\begin{enumerate}
\item $T$ has retractional skeleton by using Theorem \ref{CarattTree};
\item $T$ does not contain any copy of $[0,\omega_2]$, since the character of $T$ is equal to $\omega_1$;
\item nevertheless $T$ is not Valdivia.
\end{enumerate}
Let us finally prove point $3$. Suppose by contradiction that $T$ is Valdivia. Let $D$ be a dense $\Sigma$-subset of $T$. By Proposition \ref{Proponlyif} necessarily $D=\{t\in T:\mbox{cf}(t)\leq\omega\}$. By Theorem \ref{T0SeparatingFamily}, let $\mathcal{U}$ be a $T_0$-separating family of clopen sets witnessing that $D$ is a $\Sigma$-subset. We may assume that the elements of $\mathcal{U}$ are of the form $W_{s}^{F}$ where $s\in T$ is on a successor level and $F\subset T$ is a finite. We observe that if $W_{s}^{F}\in \mathcal{U}$ and $W_{s}^{F}\cap \operatorname{Lev}_{\omega_1}(T)\neq\emptyset$, then $t+1\in W_{s}^{F}$ for all but finitely many $t\in W_{s}^{F}\cap \operatorname{Lev}_{\omega_1}(T) $.\\
Fix $t\in \operatorname{Lev}_{\omega_1}(T)$, since $t+1\in D$ and $\mathcal{U}$ is point countable on $D$, there exists $\theta(t)<\omega_1$ such that $t+1 \notin U$ whenever $U\in \mathcal{U}(t)$ and $U\cap \operatorname{Lev}_{\theta(t)}(T)=\emptyset$. By previous observation the mapping $\theta:\operatorname{Lev}_{\omega_1}(T)\to \omega_1$ is well defined.\\
We observe that $\operatorname{Lev}_{\omega_1}(T)$ with the subspace topology is homeomorphic to $2^{\omega_1}$ endowed with the bounded topology. Since $\operatorname{Lev}_{\omega_1}(T)=\bigcup_{\alpha<\omega_1} \theta^{-1}(\alpha)$, there exists, by Proposition \ref{BaireCategory}, a $\alpha<\omega_1$ such that $\theta^{-1}(\alpha)\subset \operatorname{Lev}_{\omega_1}(T)$ is somewhere dense. Let $A$ be an open subset of $T$ such that $A\cap\theta^{-1}(\alpha)$ is dense in $A\cap \operatorname{Lev}_{\omega_1}(T)$. Since $\alpha<\omega_1$, there exists an open set $V_{t_0}\subset A$, where $\alpha<\mbox{ht}(t_{0},T)<\omega_1$ and ht$(t_0,T)$ is a successor, such that $V_{t_0}\cap\theta^{-1}(\alpha)$ is dense in $V_{t_0}\cap \operatorname{Lev}_{\omega_1}(T)$.\\
\textbf{Claim:} there exists an element $U_0$ of the family $\mathcal{U} $ such that $U_0\subset  V_{t_0}$.\\
Therefore, since $U_0$ is a clopen subset of $T$ contained in $V_{t_0}$ and $\theta^{-1}(\alpha)$ is dense in $V_{t_0}$, the intersection $U_0\cap \theta^{-1}(\alpha)$ is an infinite subset of $\operatorname{Lev}_{\omega_1}(T)$. Further we have that ht$(t_{0},T)>\alpha$, hence for every $s\in U_{0}\cap \theta^{-1}(\alpha)$ we have $s+1\notin U_0$, a contradiction.\\
It remains to prove the claim. Since $\mathcal{U}$ is point countable on $D$, in particular at $t_{0}$, and $| V_{t_0}\cap \theta^{-1}(\alpha)|>\omega$, there exist a $w\in V_{t_0}\cap \theta^{-1}(\alpha)$ and a $W_{s_0}^{F}\in \mathcal{U}$ such that $w\in W_{s_0}^{F}$ and $t_0\notin W_{s_0}^{F} $. It is enough to prove that $t_0<s_0$. Since $t_0\notin W_{s_0}^{F}$, we have that $t_0\neq s_0$, therefore suppose that $s_0<t_0$. Then there exists $r\in F$, on a successor level, such that $s_0<r\leq t_0$. Since $t_0<w$, we have that $w\notin W_{s}^{F}$, a contradiction. Thus $t_0<s_0$ and therefore $W_{s_{0}}^{F}\subset V_{t_0}$.
\end{example}

\begin{remark}
We have assumed that every tree was rooted, the above results can be proved also if the tree $T$ has finitely many minimal elements. In fact if $T$ has finitely many minimal elements, then it can be viewed as the topological direct sum of rooted trees.
\end{remark}

\noindent
\textbf{Acknowledgements:} The author is grateful to Ond\v{r}ej Kalenda for many helpful discussions. Moreover the author is grateful to Luca Motto Ros for pointing out the paper \cite{AndrettaMottoRos} and for a useful discussion on it.

\noindent
Jacopo Somaglia\\

Dipartimento di Matematica\\
Universit\`{a} degli studi di Milano\\
Via C. Saldini, 50\\
20133 Milano MI, Italy\\
\\
Department of Mathematical Analysis,\\
Faculty of Mathematics and Physics,\\
Charles University,\\ 
Sokolovsk\'{a} 83,\\
186 75 Praha 8, Czech Republic\\
\\
jacopo.somaglia@unimi.it\\
ph: 02 50316167\\

\end{document}